\documentclass[12pt]{amsart}
\usepackage{graphicx, tikz}
\usepackage{amssymb, amsthm, amsfonts, latexsym, epsfig, amsmath, amscd, amsbsy}
\usepackage{enumerate}
\usepackage{color}
 
\theoremstyle{plain}
\newtheorem{theorem}{Theorem}[section]
\newtheorem{lem}[theorem]{Lemma}
\newtheorem{cor}[theorem]{Corollary}

\newtheorem{rem}[theorem]{Remark}

\numberwithin{equation}{section}

\def\N{{\mathbb N}}
\def\Z{{\mathbb Z}}

\def\ch{\mathfrak{C}}

\def\IL{\underleftarrow\lim([\tilde c_2,\tilde c_1],T)}
\def\ILQ{\underleftarrow\lim([c_2,c_1],Q)}
\def\il{\underleftarrow\lim \,}

\def\orb{\mbox{orb}}

\def\diam{\mathop\mathrm{diam}}
\def\mesh{\mathop\mathrm{mesh}}
\def\htop{h_{top}}

\def\eps{\varepsilon}

\begin{document}

\title{On isotopy of self-homeomorphisms of quadratic
inverse limit spaces}
\author{H.\ Bruin and S.\ \v{S}timac}
\thanks{S.\ \v{S}timac was supported in part by the NEWFELPRO Grant No.~24 HeLoMa, and in
	part by the Croatian Science Foundation grant IP-2014-09-2285.}

\subjclass[2000]{54H20, 37B45, 37E05}
\keywords{isotopy, inverse limit space, tent map, quadratic map}

%\date{\today}

\begin{abstract}
We prove that every self-homeomorphism on the inverse limit space of 
a quadratic map is isotopic to some power of the shift map.
\end{abstract}

\maketitle

\baselineskip=18pt

\section{Introduction}\label{sec:intro}

The two most prominent families of unimodal maps are the family of quadratic maps 
$Q_a$, $a \in [1, 4]$, and the family of tent maps $T_s$, $s \in [1, 2]$. The inverse 
limit spaces of quadratic and tent maps share a lot of common properties. For example, 
if $f$ is a map from one of these families, then $0$ is a fixed point of $f$, the 
point $\bar 0 := (\dots, 0,0,0)$ is contained in $\il ([0, 1], f)$ and is an end-point. 
The arc-component $C$ of $\il ([0, 1], f)$ which contains $\bar 0$
is a ray converging to, but (provided $a < 4$ and $s < 2$) disjoint from the inverse limit of the core 
$\il ([c_2, c_1],f)$, and $\il ([0, 1], f) = C \cup \il ([c_2, c_1],f)$, where the 
critical or turning point is denoted as $c$ and $c_k := f^k(c)$. If $c$ is periodic 
with (prime) period $N$, then $\il ([c_2, c_1], f)$ contains $N$ end-points.

The relationships between quadratic and tent maps, and between their inverse limits, are 
mostly well understood. Each quadratic map $Q_a$ with positive topological entropy 
is semi-conjugate to a tent map $T_s$ with $\log s = h_{top}(Q_a)$, and this 
semi-conjugacy collapses (pre)periodic intervals to points, \cite{MT}. If a quadratic
map is not renormalizable and does not have an attracting periodic point then this
semi-conjugacy is a conjugacy indeed. 

It is clear that if two interval maps are topologically conjugate, then their inverse 
limit spaces are homeomorphic. The effect of renormalization on the structure of the 
inverse limit is also well-understood, \cite{BaDi3}: it produces proper subcontinua that 
are periodic under the shift homeomorphism and homeomorphic with the inverse limit space
of the renormalized map.

A very interesting question is to characterize groups of homeomorphisms which act on  
inverse limits of unimodal maps. In \cite{BS} we proved that for every homeomorphism 
$h : \il ([0, 1], T_s) \to \il ([0, 1], T_s)$ there exists $R \in \Z$ such that $h$ is 
isotopic to $\sigma^R$, where $\sigma$ is the standard shift map (see also \cite{BJKK, BKRS}). 
Thus it is natural to ask the same question for the `fuller' quadratic family, which includes 
(infinitely) renormalizable maps. The answer does not follow in a straightforward way from 
\cite{BS}. Some work should be done, and this is what we have proved in this paper:
\begin{theorem}\label{main} Let  $H : \il ([0, 1], Q) \to \il ([0, 1], Q)$ 
	be a homeomorphism. Then $H$ is isotopic to $\sigma^R$ for some $R \in \Z$.	
\end{theorem}

The paper is organized as follows. Section~\ref{sec:prelim} gives basic definitions. Section~\ref{sec:PI}
gives the major step for the isotopy result from tent maps inverse limits to quadratic
maps inverse limits. In Section~\ref{sec:I} we show how homeomorphisms act on $p$-points and prove 
our main theorem. These last proofs depend largely on the results obtained in \cite{BS}
and \cite{BBS}. Finally, in Section~\ref{sec:E} we use our main theorem to calculate topological
entropy of self-homeomorphisms on inverse limits of quadratic maps.

\section{Preliminaries}\label{sec:prelim}

Let $\N = \{ 1,2,3,\dots\}$ be the set of natural numbers and $\N_0 = \N \cup \{ 0 \}$. 
We consider two families of unimodal maps, the family of quadratic maps 
$Q_a : [0,1] \to [0,1]$, with $a \in [1, 4]$, defined as $Q_a(x) = ax(1-x)$, and the 
family of tent maps $T_s:[0,1] \to [0,1]$ with slope $\pm s$, $s \in [1, 2]$, defined 
as $T_s(x) = \min\{sx, s(1-x)\}$. Let $f$ be a map form any of these two families. The 
{\em critical} or {\em turning} point is $c := 1/2$. Write $c_k := f^k(c)$. The closed 
$f$-invariant interval $[c_2, c_1]$ is called the {\em core}.

The inverse limit space $\il ([0, 1], f)$ is the collection of all backward orbits
\[
\{ x = (\dots, x_{-2}, x_{-1}, x_0) :  f(x_{-i-1}) = x_{-i} \in [0,c_1]
\textrm{ for all } i \in \N_0 \},
\]
equipped with metric $d(x,y) = \sum_{n \le 0} 2^n |x_n - y_n|$ and {\em induced}, or 
{\em shift homeo\-morphism}
\[
\sigma(x) :=  \sigma_{f}(\dots, x_{-2}, x_{-1}, x_0) = (\dots, x_{-2}, x_{-1}, x_0, f(x_0)).
\]
Let $\pi_k: \il ([0, 1], f) \to [0, c_1]$, $\pi_k(x) = x_{-k}$ be the $k$-th projection 
map. For any point $x \in \il ([0, 1], f)$, the {\em composant} of $x$ in $\il ([0, 1], f)$ is 
the union of all proper subcontinua of $\il ([0, 1], f)$ containing $x$, and the 
{\em arc-component} of $x$ in $\il ([0, 1], f)$ is the union of all arcs in $\il ([0, 1], f)$ 
containing $x$.

We review some of the main tools introduced in \cite{BBS} and which are necessary here as 
well. We define {\em $p$-points} as those points
$x = (\dots, x_{-2}, x_{-1}, x_0) \in \il ([0, 1], f)$  such that $x_{-p-k} = c$ for some
$k \in \N_0$. The supremum $k$ of the set of integers with this property is called the {\em $p$-level} of
$x$, $L_p(x) := k$. Note that $k$ can be $\infty$, and this will happen if, for example, the turning point is periodic, say of period $N$. In this case the 
corresponding inverse limit will have $N+1$ end-points: One is $\bar 0$ 
and the others are $p$-points with $p$-level $\infty$ for every $p$.

Among the $p$-points of $C$ there are special ones, called {\em salient}, which are 
center points of symmetries in $C$. Homeomorphisms preserve these symmetries to such an extent 
that it is possible to prove that salient points map close to salient points.

We call a $p$-point $y \in C$ \emph{salient} if $0 \le L_p(x) < L_p(y)$ for every $p$-point
$x \in (\bar 0 , y)$. Let $(s_p^i)_{i \in \N}$ be the sequence of all salient $p$-points of 
$C$, ordered such that $s_p^i \in (\bar 0 , s_p^{i+1})$ for all $i \ge 1$.
Since by definition $L_p(s_p^i) > 0$, for all $i \ge 1$, we have $L_p(s_p^1) = 1$. Also, since
$s_p^i = \sigma^{i-1}(s_p^1)$, we have $L_p(s_p^i) = i$, for every $i \in \N$. Therefore, for 
every $p$-point $x$ of $\il ([0, 1], f)$ with $L_p(x) \ne 0$, there exists a unique salient 
$p$-point $s_p^k$ such that $L_p(x) = L_p(s_p^k) = k$. Note that the salient $p$-points depend 
on $p$: if $p \ge q$, then the salient $p$-point $s_p^i$ equals the salient $q$-point 
$s_q^{i+p-q}$.

A continuum is {\em chainable} if for every $\eps > 0$, there is a cover 
$\{ \ell^1, \dots , \ell^n\}$ of open sets (called {\em links}) of diameter less than $\eps$ such that 
$\ell^i \cap \ell^j \neq \emptyset$ if and only if $|i-j| \le 1$. Such a cover is called a 
{\em chain}. Clearly the interval $[0,c_1]$ is chainable. Throughout, we will use sequence of 
chains $\ch_p$ of $\il ([0, 1], f)$ satisfying the following properties:
\begin{enumerate}
	\item there is a chain $\{ I^1_p, \dots , I^n_p \}$ of $[0,c_1]$ such that
	$\ell^j_p := \pi_p^{-1}(I^j_p)$ are the links of $\ch_p$;
	\item each point $x \in \bigcup_{i=0}^p f^{-i}(c)$ is a boundary point of some link $I^j_p$;
	\item for each $i$ there is $j$ such that $f(I^i_{p+1}) \subset I^j_p$.
\end{enumerate}
If $\max_j |I^j_p| < \eps s^{-p}/2$ then 
$\mesh(\ch_p) := \max\{ \diam(\ell_p) :\ell_p \in \ch_p\} < \eps$, which shows that 
$\il ([0, 1], f)$ is indeed chainable. Condition (3) ensures that $\ch_{p+1}$ \emph{refines}
$\ch_p$ (written $\ch_{p+1} \preceq \ch_p$).

Note that all $p$-points of $p$-level $k$ belong to the same link of $\ch_p$. (This follows
by property (1) of $\ch_p$, because $L_p(x) = L_p(y)$ implies $\pi_p(x) = \pi_p(y)$.) 
Therefore, every link of $\ch_p$ which contains a $p$-point of $p$-level $k$, contains also 
the salient $p$-point $s_p^k$.

Let $\ell^0, \ell^1, \dots, \ell^k$ be those links in $\ch_p$ that are successively visited 
by an arc $A \subset C$ (hence $\ell^i \neq \ell^{i+1}$, $\ell^i \cap \ell^{i+1} \neq \emptyset$ 
and $\ell^i = \ell^{i+2}$ is possible if $A$ turns in $\ell^{i+1}$). We call the arc $A$ 
{\em $p$-link-symmetric} if $\ell^i = \ell^{k-i}$ for $i = 0, \dots, k$, and {\em maximal 
$p$-link-symmetric} if it is $p$-link-symmetric and there is no $p$-link-symmetric arc $B \supset A$ 
and passing through more links than $A$. In any of these cases, $k$ is even and the link $\ell^{k/2}$ 
is called the {\em central link} of $A$.

For $x \in X$, define the {\em itinerary} $i(x) = \{i_n\}_{n \in Z}$ as the sequence where
$$
i_n(x) = \begin{cases}
0 & x_n \le c, \\
1 & x_n \ge c,
\end{cases}
$$
and if $x_n = c$, we write $i_n(x) = \substack{0\\1}$.

\section{Pseudo-isotopy}\label{sec:PI}

Let $Q$ be a quadratic map of entropy $h_{top}(Q) > \frac12\log 2$, which is 
renormalizable, but  (due to the entropy restraint) of period $N > 2$. Let 
$\{J_j\}_{j=0}^{N-1}$ be the periodic cycle of intervals numbered so that 
$c \in J_0$ and $Q(J_i) = J_{i+1 \pmod N}$. Denote $J = \bigcup_{j=0}^{N-1} J_j$.
Also assume that $J_i = [p_i, \hat p_i]$, where $p_i$ is $N$-periodic and 
$Q^N(\hat p_i) = p_i$. Let $T := T_s$ be the semiconjugate tent map with 
$s = \exp(h_{top}(Q))$. Then $T$ has an $N$-periodic critical point.

Let $X := \ILQ$ and $\tilde X := \IL$. Let 
$G_j := \{ x \in X : \pi_k(x) \in J_{j-k \pmod N}, k \in \N_0 \}$, 
$j \in \{ 0, \dots , N-1 \}$. These are the (maximal) proper subcontinua of $X$ that 
are not arcs or points. On the other hand, all proper subcontinua of 
$\tilde X$ are arcs and points and $\tilde X$ has $N$ endpoints $e_j$, 
where $\tilde\pi_0(e_j) = \tilde c_j$. Here $\pi_i:X \to [c_2,c_1]$ and 
$\tilde \pi_i : \tilde X \to [\tilde c_2, \tilde c_1]$ are the coordinate 
projections.

There is a unique arc-component $Z_j$ of $X \cap \pi_0^{-1}(J_j)$ that 
compactifies exactly on $G_j$. This is a ray, and we can extent it on
one side with an arc $Z^*_j$ such that $\pi_0(Z^*_j) = [r_{j,0}, p_j]$
(where the point $r_{j,0}$ close to $p_j$ is chosen below).

\begin{theorem}\label{thm:cont}
There exists a continuous onto map $\phi : \ILQ \to \IL$ such that 
$\phi(G_j) = e_j$ and $\phi$ is one-to-one on 
$\ILQ \setminus \bigcup_{j=0}^{N-1} G_j$.
\end{theorem}

\begin{proof}
For the quadratic map, let $(r_{j,k})_{k \in \N_0}$ be a monotone sequence 
of points such that $r_{j,k}$ belong to a single component of 
$[c_2,c_1] \setminus J$, and $r_{j,k} \to p_j$ (see Figure~\ref{fig:r}).

\begin{figure}[ht]
\begin{center}
\begin{tikzpicture}[scale=0.9]
 \draw[-, draw=black] (0,1) -- (3,1); \draw[-, draw=black] (3,0.9) -- (3,1.1);
\draw[-, draw=black] (0,0.9) -- (0,1.1);  \node at (3,0.5) {\tiny $p_2$};
\draw[-, draw=black] (1,0.9) -- (1,1.1);  \node at (1,0.5) {\tiny $c_2$};
  \node at (0,0.5) {\tiny $\hat p_2$}; \node at (1.5,1.3) {\tiny $J_2$};
\draw[-, draw=black] (10,0.9) -- (10,1.1);  \node at (10,0.5) {\tiny $p_0$};
\draw[-, draw=black] (6,0.9) -- (6,1.1);  \node at (6,0.5) {\tiny $\hat p_0$};
 \draw[-, draw=black] (6,1) -- (10,1); \draw[-, draw=black] (8,0.9) -- (8,1.1);
  \node at (8,0.5) {\tiny $c$}; \node at (8,1.3) {\tiny $J_0$};
\draw[-, draw=black] (13,0.9) -- (13,1.1);  \node at (13,0.5) {\tiny $p_1$};
  \draw[-, draw=black] (13,1) -- (16,1); \draw[-, draw=black] (16,0.9) -- (16,1.1);
\draw[-, draw=black] (15,0.9) -- (15,1.1);  \node at (15,0.5) {\tiny $c_1$};
  \node at (16,0.5) {\tiny $\hat p_1$}; \node at (14.5,1.3) {\tiny $J_1$}; 
 \node at (4,1) {\tiny $\bullet$}; \node at (4,0.5) {\tiny $r_{2,0}$};
 \node at (3.6,1) {\tiny $\bullet$}; % \node at (3.6,0.5) {\tiny $r_{2,1}$};
 \node at (3.35,1) {\tiny $\bullet$}; % \node at (3.3,0.5) {\tiny $r_{2,2}$};
 \node at (3.15,1) {\tiny $\bullet$}; % \node at (3.15,0.5) {\tiny $r_{2,3}$};
 \node at (11,1) {\tiny $\bullet$}; \node at (11,0.5) {\tiny $r_{0,0}$};
 \node at (10.6,1) {\tiny $\bullet$}; % \node at (10.6,0.5) {\tiny $r_{0,1}$};
 \node at (10.35,1) {\tiny $\bullet$}; % \node at (10.3,0.5) {\tiny $r_{0,2}$};
 \node at (10.15,1) {\tiny $\bullet$}; % \node at (10.15,0.5) {\tiny $r_{0,3}$};
 \node at (12,1) {\tiny $\bullet$}; \node at (12,0.5) {\tiny $r_{1,0}$};
 \node at (12.4,1) {\tiny $\bullet$}; % \node at (12.4,0.5) {\tiny $r_{1,1}$};
 \node at (12.65,1) {\tiny $\bullet$}; % \node at (12.65,0.5) {\tiny $r_{1,2}$};
 \node at (12.85,1) {\tiny $\bullet$}; % \node at (12.85,0.5) {\tiny $r_{1,3}$};
\end{tikzpicture}
\caption{The intervals $J_j$ with sequences $r_{j,k} \to p_j$.}
\label{fig:r}
\end{center}
\end{figure}
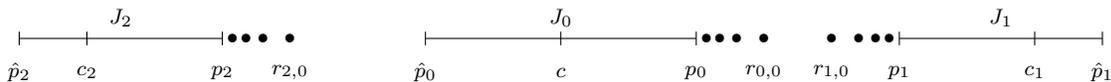 

Without loss of generality, we can set $Q(r_{j,k}) = r_{j+1,k}$ for $0 \leq j < N-1$ 
and $Q(r_{N-1,k+1}) = r_{0,k}$ for $k \geq 0$. This means that the sequence 
$(r_{j,k} : j = 0, \dots N-1, k \in \N_0)$, starting from $r_{0,0}$, forms a single 
backward orbit.

Similarly, for the tent map $T$, for fixed $0 \leq j < N$, let $(\tilde r_{j,k})_{k \in \N_0}$ be a monotone sequence 
of points such that $\tilde r_{j,k}$ belong to a single component of 
$[\tilde c_2,\tilde c_1] \setminus \orb_T(\tilde c)$, and $\tilde r_{j,k} \to \tilde c_j$. 
Again we set $T(\tilde r_{j,k}) = \tilde r_{j+1,k}$ for $0 \leq j < N-1$, and 
$T(\tilde r_{N-1,k+1}) = \tilde r_{0,k}$ for $k \geq 0$. 

Note that if $x \in X$ is such that $\pi_n(x) \notin J$ for all $n \in \Z$, then there is 
a unique point $\tilde x \in \tilde X$ such that $x$ and $\tilde x$ have the same itinerary.
Without loss of generality, we can indeed assume that $\orb(r_{0,0}) \cap J = \emptyset$, 
and then indeed choose $\tilde r_{0,0}$ with the same itinerary as $r_{0,0}$. Then $r_{j,k}$ 
and $\tilde r_{j,k}$ have the same itinerary for every $j,k$.

Now take $x \in X$. Depending on whether $\pi_0(x) \in \bigcup_j (r_{j,0}, \hat p_j)$ or not, 
and on whether $G_j$ are the Knaster continua or not, we have different algorithms to define $\phi(x)$.

{\bf (1) The case $\pi_0(x) \notin \bigcup_j (r_{j,0}, \hat p_j)$.}
Suppose that $x$ belongs to a component $W$ of 
$X \setminus \pi_0^{-1} (\, \bigcup_j (r_{j,0}, \hat p_j))$. There is a unique component $\tilde W$ 
of $\tilde X \setminus \tilde\pi_0^{-1}(\, \bigcup_j (\tilde r_{j,0} , \tilde c_j) )$ which has the 
same backward itinerary as $W$. Define $\phi:W \to \tilde W$ such that 
$\tilde\pi_0 \circ \phi \circ \pi^{-1}$ is an affine map from $\pi_0(W)$ to $\tilde\pi_0(\tilde W)$.

{\bf (2) The case $\pi_0(x) \in \bigcup_j (r_{j,0}, \hat p_j)$, non-Knaster construction.}
Assume that $\pi_0(x) \in [r_{j,0}, \hat p_j]$ for some $j$, and let $W$ be the component of 
$\pi_0^{-1}( [r_{j,0}, \hat p_j])$ containing $x$, and $V \subset W$ be the corresponding component 
of $\pi_0^{-1}( [p_j, \hat p_j])$. Let 
$$
\Lambda = \sup\{ L(y) : y \in V \text{ is a $p$-point}\},
$$
where $L := L_p$ is the $p$-level of a $p$-point. The definition of $\phi|_W$ will
depend on the value of $\Lambda$ (see Figure~\ref{fig:phinK}).
\begin{enumerate}[(2.1)]
\item
If $\Lambda = 0$, then $\pi_0:W \to [r_{j,0}, \hat p_j]$ is injective. Let $\tilde W \subset \tilde X$ 
be the unique component of $\tilde\pi_0^{-1}([\tilde r_{j,0}, \tilde c_j])$ that has the same backward 
itinerary as $W$. Define $\phi:W \to \tilde W$ to be the homeomorphism such that 
$\tilde\pi_0 \circ \phi \circ \pi^{-1}$ is the affine map from $[r_{j,0}, \hat p_j] = \pi_0(W)$ onto 
$[\tilde r_{j,0}, \tilde c_j] = \tilde\pi_0(\tilde W)$ so that 
$\tilde\pi_0 \circ \phi \circ \pi^{-1}(r_{j,0}) = \tilde r_{j,0}$.
\item 
If $0 < \Lambda < \infty$, then let $m \in W$ be the $p$-point of level $L(m) = \Lambda$. In fact, $m$ 
is the unique point with this property: it is the midpoint of $W$. Furthermore, there are two finite 
sequences of points $\{ v_k\}_{k=1}^\lambda$ and $\{ \hat v_k\}_{k=1}^\lambda$, $\lambda \leq \Lambda$,
inside $V$ such that
$\partial V = \{ v_1, \hat v_1 \}$, $v_\lambda = \hat v_\lambda = m$, $v_{k+1}$ is the $p$-point in 
$[v_k, m]$ such that $L(v_{k+1}) > L(v_k)$ and no point in $(v_k, v_{k+1})$ has level larger than 
the level of $v_k$, and $\hat v_{k+1}$ is the $p$-point in $[\hat v_k, m]$ such that 
$L(\hat v_{k+1}) > L(\hat v_k)$ and no point in $(\hat v_k, \hat v_{k+1})$ has level larger than 
the level of $\hat v_k$.
(Note that if $V$ and $V'$ are two different components of $\pi^{-1}_0([p_j, \hat p_j])$
with levels of their midpoints satisfying $\Lambda < \Lambda'$,
then $\lambda < \lambda'$, and in fact, the levels of the points $v'_k \in V'$ are a superset
of the levels of the points $v_k \in V$.)
	
Let $\tilde W \subset \tilde X$ be the component of $\tilde\pi_0^{-1}([\tilde r_{j,0}, \tilde c_j]$ 
so that the two endpoints of $\tilde W$ have the same itineraries as the corresponding endpoints of $W$.
Note that the midpoint $\tilde m$ of $\tilde W$ has the level $L(\tilde m) = \Lambda$.

Define $\phi:W \to \tilde W$ 
to be the homeomorphism such that
\begin{enumerate}[(a)]
\item $\tilde\pi_0 \circ \phi \circ \pi^{-1}$ maps $[r_{j,0}, p_j]$ affinely onto 
$[\tilde r_{j,0}, \tilde r_{j,1}]$;
\item $[v_k, v_{k+1}]$ and $[\hat v_k, \hat v_{k+1}]$ are mapped to the two components of 
$\tilde\pi_0^{-1}([\tilde r_{j,k}, \tilde r_{j,k+1}]) \cap \tilde W$, for $0 \leq k < \lambda-1$.
\item $[v_{\lambda-1}, \hat v_{\lambda-1}]$ is mapped onto 
$\tilde\pi_0^{-1}([\tilde r_{j,\lambda-1}, \tilde c_j]) \cap \tilde W$ in such a way that 
$\tilde\pi_0 \circ \phi(m) = \tilde c_j$.
\end{enumerate}

\item If $\Lambda = \infty$ and $x \in Z_j$ (but $Z_j \cap G_j = \emptyset$ since we are in the 
non-Knaster case), then $V = Z_j \subset W$ is a ray, and we can define an infinite sequence
$\{ v_k \}_{k=1}^\infty$ so that $\pi_0(v_1) = p_j$ and $v_{k+1}$ is the $p$-point on 
$Z_j \setminus [v_1, v_k]$ such that $L(v_{k+1}) > L(v_k)$ and no $p$-point on $(v_k, v_{k+1})$ 
has the level larger than the level of $v_k$.

Let $\tilde W$ be the component of $\tilde\pi_0^{-1}([\tilde r_{j,0}, \tilde c_j))$ having $e_j$ 
as boundary point. Define $\phi:W \to \tilde W$ to be the homeomorphism such that
\begin{enumerate}[(a)]
\item $\tilde\pi_0 \circ \phi \circ \pi^{-1}$ maps $[r_{j,0}, p_j]$ affinely onto 
$[\tilde r_{j,0}, \tilde r_{j,1}]$;
\item $[v_k, v_{k+1}]$ is mapped to 
$\tilde\pi_0^{-1}([\tilde r_{j,k}, \tilde r_{j,k+1}]) \cap \tilde W$, for $k \geq 1$.
\end{enumerate}
\item $\phi(x) = e_j$ for every $x \in G_j$.
\end{enumerate}

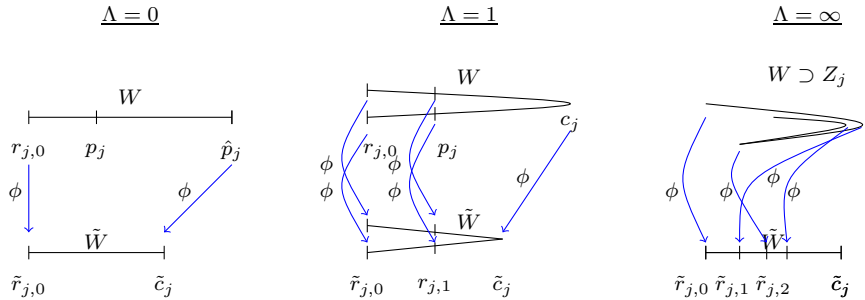
\begin{figure}[ht]
\begin{center}
\begin{tikzpicture}[scale=0.9]
\node at (1.5,4.5) {\tiny $\underline{\Lambda = 0}$};
\node at (1.5,3.3) {\tiny $W$};\node at (1,1.2) {\tiny $\tilde W$};
\draw[-, draw=black] (0,3) -- (3,3); \draw[-, draw=black] (3,2.9) -- (3,3.1);
\draw[-, draw=black] (0,2.9) -- (0,3.1);  \node at (0,2.5) {\tiny $r_{j,0}$};
\draw[-, draw=black] (1,2.9) -- (1,3.1);  
\node at (1,2.5) {\tiny $p_j$}; \node at (3,2.5) {\tiny $\hat p_j$}; 
\draw[-, draw=black] (3,2.9) -- (3,3.1);  
  \node at (0,0.5) {\tiny $\tilde r_{j,0}$};
\draw[-, draw=black] (0,1) -- (2,1); \draw[-, draw=black] (2,0.9) -- (2,1.1);
\draw[-, draw=black] (0,0.9) -- (0,1.1);  \node at (2,0.5) {\tiny $\tilde c_j$};
\draw[->, draw=blue] (0,2.3) -- (0,1.3);  \node at (-0.2,1.9) {\tiny $\phi$};
\draw[->, draw=blue] (3,2.3) -- (2,1.3);   \node at (2.3,1.9) {\tiny $\phi$};
%%%
\node at (6.5,4.5) {\tiny $\underline{\Lambda = 1}$};
\node at (6.5,3.6) {\tiny $W$};\node at (6.5,1.5) {\tiny $\tilde W$};
\draw[-] (5,3) .. controls (9, 3.2) .. (5,3.4);
\draw[-, draw=black] (5,2.9) -- (5,3.1); 
\draw[-, draw=black] (5,3.3) -- (5,3.5);  \node at (5.2,2.5) {\tiny $r_{j,0}$};
\draw[-, draw=black] (6,2.95) -- (6,3.15);  \draw[-, draw=black] (6,3.25) -- (6,3.45);  
\node at (6.2,2.5) {\tiny $p_j$}; \node at (8,2.9) {\tiny $c_j$}; 
 \node at (5,0.5) {\tiny $\tilde r_{j,0}$};
\draw[-, draw=black] (5,1) -- (7,1.2); \draw[-, draw=black] (5,1.4) -- (7,1.2); 
\draw[-, draw=black] (5,0.9) -- (5,1.1);\draw[-, draw=black] (5,1.3) -- (5,1.5);
\node at (7,0.5) {\tiny $\tilde c_j$};
\draw[-, draw=black] (6,0.98) -- (6,1.15); \draw[-, draw=black] (6,1.25) -- (6,1.42);
\node at (6,0.5) {\tiny $r_{j,1}$};
%\draw[->, draw=black] (5,2.3) -- (5,1.3);
\draw[->, draw=blue] (5,2.75) .. controls (4.5, 2) .. (5,1.15); \node at (4.4,1.9) {\tiny $\phi$};
\draw[->, draw=blue] (5,3.25) .. controls (4.5, 2.4) .. (5,1.55); \node at (4.4,2.3) {\tiny $\phi$};
\draw[->, draw=blue] (8,2.8) -- (7,1.3);   \node at (7.3,2.1) {\tiny $\phi$}; 
\draw[->, draw=blue] (6,2.9) .. controls (5.5, 2) .. (6,1.15); \node at (5.4,1.9) {\tiny $\phi$};
\draw[->, draw=blue] (6,3.25) .. controls (5.5, 2.4) .. (6,1.55); \node at (5.4,2.3) {\tiny $\phi$};
%%%
%%%
\node at (11.5,4.5) {\tiny $\underline{\Lambda = \infty}$};
\node at (11.5,3.6) {\tiny $W \supset Z_j$};\node at (11,1.2) {\tiny $\tilde W$};
\draw[-] (10.5,2.6) .. controls (13, 2.9) .. (10,3.2);
\draw[-] (10.5,2.6) .. controls (12.5, 2.9) .. (11,3);
\draw[->, draw=blue] (10,3) .. controls (9.55, 2.1) .. (10,1.15); \node at (9.5,1.9) {\tiny $\phi$};
\draw[->, draw=blue] (12.3,2.85) .. controls (10.5, 2.1) .. (10.5,1.15); \node at (11,2.1) {\tiny $\phi$};
\draw[->, draw=blue] (10.5,2.5) .. controls (10.3, 2.1) .. (10.9,1.15); \node at (10.3,1.9) {\tiny $\phi$};
\draw[->, draw=blue] (12.1,2.85) .. controls (11.1, 2.1) .. (11.2,1.15); \node at (11.3,1.9) {\tiny $\phi$};
\draw[-, draw=black] (10,1) -- (12,1); \draw[-, draw=black] (12,0.9) -- (12,1.1);
\draw[-, draw=black] (10,0.9) -- (10,1.1);  \node at (12,0.5) {\tiny $\tilde c_j$};
\node at (9.8,0.5) {\tiny $\tilde r_{j,0}$};
\draw[-, draw=black] (10.5,0.9) -- (10.5,1.1); \node at (10.4,0.5) {\tiny $\tilde r_{j,1}$};
\draw[-, draw=black] (10.9,0.9) -- (10.9,1.1); \node at (11,0.5) {\tiny $\tilde r_{j,2}$};
\draw[-, draw=black] (11.2,0.9) -- (11.2,1.1); %\node at (11.2,0.5) {\tiny $\tilde r_{j,2}$};
\draw[-, draw=black] (10,1) -- (12,1); \draw[-, draw=black] (12,0.9) -- (12,1.1);
\draw[-, draw=black] (10,0.9) -- (10,1.1);  \node at (12,0.5) {\tiny $\tilde c_j$};
\end{tikzpicture}
\caption{The non-Knaster case: the arcs $W$ and their images under $\phi$.
The labels refer to the $\pi_0$-images of the points.}
\label{fig:phinK}
\end{center}
\end{figure} 

{\bf (3) The case $\pi_0(x) \in \bigcup_j (r_{j,0}, \hat p_j)$, Knaster construction.}
Now we adapt the construction for the case that the renormalization $Q^N|_{J_0}$ is a full unimodal 
map (i.e., $G_j$ is the Knaster continuum). In this case $Z_j \subset G_j$ and the construction of 
$\phi:W \to \tilde W$ for item (2.1) remains the same. Item (2.2) is changed into (3.2) below, and items 
(2.3) and (2.4) are combined into item (3.4), 
\begin{itemize}
\item[(3.2)]
With $W, \tilde W$ and sequence $\{ v_k\}_{k=1}^\lambda$ and $\{ \hat v_k\}_{k=1}^\lambda$ with 
$v_\lambda = \hat v_\lambda = m$ as before, define $\phi:W \to \tilde W$ to be the homeomorphism such 
that
\begin{enumerate}[(a)]
\item $\tilde\pi_0 \circ \phi \circ \pi^{-1}$ maps $[r_{j,0}, p_j]$ affinely onto 
$[\tilde r_{j,0}, \tilde r_{j,\lambda + 1}]$;
\item $[v_k, v_{k+1}]$ and $[\hat v_k, \hat v_{k+1}]$ are mapped to the two components of 
$\tilde\pi_0^{-1}([\tilde r_{j,\lambda+k}, \tilde r_{j,\lambda+k+1}]) \cap \tilde W$,
for $1 \leq k < \lambda-1$.
\item $[v_{\lambda-1}, \hat v_{\lambda-1}]$ is mapped onto 
$\tilde\pi_0^{-1}([\tilde r_{j,2\lambda-1}, \tilde c_j]) \cap \tilde W$ in such a way that 
$\tilde\pi_0 \circ \phi(m) = \tilde c_j$ (see Figure~\ref{fig:phiK}).
\end{enumerate}

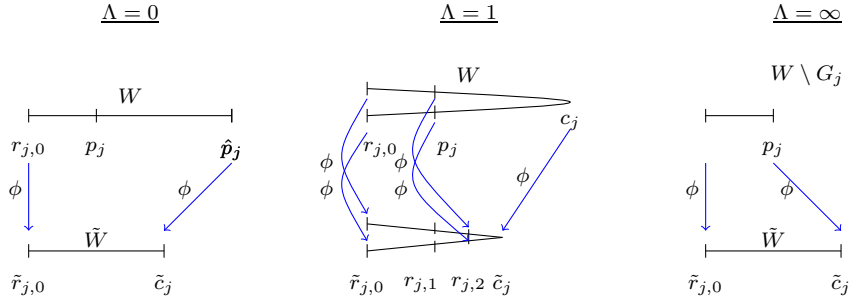
\begin{figure}[ht]
\begin{center}
\begin{tikzpicture}[scale=0.9]
\node at (1.5,4.5) {\tiny $\underline{\Lambda = 0}$};
\node at (1.5,3.3) {\tiny $W$};\node at (1,1.2) {\tiny $\tilde W$};
\draw[-, draw=black] (0,3) -- (3,3); \draw[-, draw=black] (3,2.9) -- (3,3.1);
\draw[-, draw=black] (0,2.9) -- (0,3.1);  \node at (0,2.5) {\tiny $r_{j,0}$};
\draw[-, draw=black] (1,2.9) -- (1,3.1);  
\node at (1,2.5) {\tiny $p_j$}; \node at (3,2.5) {\tiny $\hat p_j$}; 
\draw[-, draw=black] (3,2.9) -- (3,3.1);  
  \node at (0,0.5) {\tiny $\tilde r_{j,0}$};
\draw[-, draw=black] (0,1) -- (2,1); \draw[-, draw=black] (2,0.9) -- (2,1.1);
\draw[-, draw=black] (0,0.9) -- (0,1.1);  \node at (2,0.5) {\tiny $\tilde c_j$};
\draw[->, draw=blue] (0,2.3) -- (0,1.3);  \node at (-0.2,1.9) {\tiny $\phi$};
\draw[->, draw=blue] (3,2.3) -- (2,1.3);   \node at (2.3,1.9) {\tiny $\phi$};
%%%
\node at (6.5,4.5) {\tiny $\underline{\Lambda = 1}$};
\node at (6.5,3.6) {\tiny $W$};  %\node at (6.5,1.5) {\tiny $\tilde W$};
\draw[-] (5,3) .. controls (9, 3.2) .. (5,3.4);
\draw[-, draw=black] (5,2.9) -- (5,3.1); 
\draw[-, draw=black] (5,3.3) -- (5,3.5);  \node at (5.2,2.5) {\tiny $r_{j,0}$};
\draw[-, draw=black] (6,2.95) -- (6,3.15);  \draw[-, draw=black] (6,3.25) -- (6,3.45);  
\node at (6.2,2.5) {\tiny $p_j$}; \node at (8,2.9) {\tiny $c_j$}; 
 \node at (5,0.5) {\tiny $\tilde r_{j,0}$};
\draw[-, draw=black] (5,1) -- (7,1.2); \draw[-, draw=black] (5,1.4) -- (7,1.2); 
\draw[-, draw=black] (5,0.9) -- (5,1.1); \draw[-, draw=black] (5,1.3) -- (5,1.5);
\node at (7,0.5) {\tiny $\tilde c_j$};
\draw[-, draw=black] (6,0.98) -- (6,1.15); \draw[-, draw=black] (6,1.25) -- (6,1.42);
\node at (5.8,0.5) {\tiny $r_{j,1}$};
\draw[-, draw=black] (6.5,1.1) -- (6.5,1.32); %\draw[-, draw=black] (6,1.25) -- (6,1.42);
\node at (6.5,0.5) {\tiny $r_{j,2}$};
%\draw[->, draw=black] (5,2.3) -- (5,1.3);
\draw[->, draw=blue] (5,2.75) .. controls (4.5, 2) .. (5,1.15); \node at (4.4,1.9) {\tiny $\phi$};
\draw[->, draw=blue] (5,3.25) .. controls (4.5, 2.4) .. (5,1.55); \node at (4.4,2.3) {\tiny $\phi$};
\draw[->, draw=blue] (8,2.8) -- (7,1.3);   \node at (7.3,2.1) {\tiny $\phi$}; 
\draw[->, draw=blue] (6,2.9) .. controls (5.5, 2) .. (6.5,1.15); \node at (5.5,1.9) {\tiny $\phi$};
\draw[->, draw=blue] (6,3.25) .. controls (5.5, 2.4) .. (6.5,1.35); \node at (5.5,2.3) {\tiny $\phi$};
%%%
%%%
\node at (11.5,4.5) {\tiny $\underline{\Lambda = \infty}$};
\node at (11.5,3.6) {\tiny $W \setminus G_j$};\node at (11,1.2) {\tiny $\tilde W$};
\draw[-, draw=black] (10,3) -- (11,3); \draw[-, draw=black] (11,2.9) -- (11,3.1);
\draw[-, draw=black] (10,2.9) -- (10,3.1);  
\node at (11,2.5) {\tiny $p_j$}; \node at (3,2.5) {\tiny $\hat p_j$};  
\draw[-, draw=black] (10,1) -- (12,1); \draw[-, draw=black] (12,0.9) -- (12,1.1);
\draw[-, draw=black] (10,0.9) -- (10,1.1);  \node at (12,0.5) {\tiny $\tilde c_j$};
\node at (10,0.5) {\tiny $\tilde r_{j,0}$};
\draw[->, draw=blue] (10,2.3) -- (10,1.3);  \node at (9.8,1.9) {\tiny $\phi$};
\draw[->, draw=blue] (11,2.3) -- (12,1.3);   \node at (11.2,1.9) {\tiny $\phi$};
\end{tikzpicture}
\caption{The Knaster case: the arcs $W$ and their images under $\phi$.
The labels refer to the $\pi_0$-images of the points.}
\label{fig:phiK}
\end{center}
\end{figure} 

\item[(3.4)] If $\Lambda = \infty$ and $x \in Z^*_j \cup G_j$, then let $\tilde W$ 
be the component of $\tilde X$ with boundary point $e_j$ and 
$\tilde\pi_0(\tilde W) = [\tilde r_{j,0}, \tilde c_j)$. Let $\phi:Z^*_j \to \tilde W$ 
be such that 
$\tilde\pi_0 \circ \phi \circ \pi_0^{-1}:[r_{j,0}, p_j) \to [\tilde r_{j,0}, \tilde c_j)$ 
is affine and $\phi(x) = e_j$ if $x \in G_j$.
\end{itemize}

By construction, $\phi : W \to \tilde W$ is always a homeomorphism for all the components 
$W$ in this construction, and their union together with $\bigcup_{j=0}^{N-1}G_j$ is the entire 
$X$, just as the union of all $\tilde W$ together with $\bigcup_{j=0}^{N-1}\{ e_j\}$ is the entire 
$\tilde X$. Therefore $\phi$ is onto and one-to-one on $\tilde X \setminus \bigcup_{j=0}^{N-1} G_j$.

The construction of $\phi:W \to \tilde W$ essentially depends only on the value of $\Lambda$, 
in the sense that without loss of generality,
\begin{equation}\label{eq:zzz}
\tilde\pi_0 \circ \phi \circ \pi^{-1}:\pi_0(W) \to \tilde\pi_0(\tilde W)
\end{equation}
can be chosen to depend only on $\Lambda = \Lambda(W)$.
This makes $\phi$ continuous on $X \setminus \bigcup_{j=0}^{N-1}(G_j \cup Z_j)$.
Finally, in the non-Knaster case, 
$W \to Z_j \cup G_j$ in Hausdorff metric if and only if $\Lambda(W) \to \infty$
and $\Lambda(W) \pmod N = j$.
Recall that we specified $\phi|_{Z_j}$ to have the 
limit dynamics of $\phi|_W$ for $W \to Z_j \cup G_j$, thus achieving
continuity of $\phi$ on each $Z_j \cup G_j$.

In the Knaster case, $\diam(\phi(W)) \to 0$ as $W \to G_j$, so here $\phi$ is also continuous 
on each $G_j$.
\end{proof}

\begin{cor}
Given a homeomorphism $h: X \to X$, the map $\tilde h := \phi \circ h \circ \phi^{-1}$ 
is a well-defined homeomorphism on $\IL$.
\end{cor}

\begin{proof}
The end-points $e_j$ are the only points where $\phi^{-1}$ is not single-valued.
Therefore $\tilde h$ is well-defined on $\IL \setminus \{ e_0, \dots, e_{N-1}\}$.
However, since $h$ is a homeomorphism, it has to permute the subcontinua $G_j$,
and therefore $h \circ \phi^{-1}(e_j) = G_i$ for some $i$
and $\phi \circ h \circ \phi^{-1}(e_j) = e_i$.
Therefore $\tilde h$ is defined (in a single-valued way) also at the 
endpoints, and it permutes
them by the same permutation as $h$ permutes the subscontinua $G_j$.

Since $\phi$ is continuous and injective outside $G_j$, $\tilde h$ is 
continuous on $\tilde X \setminus \{ e_0, \dots, e_{N-1}\}$.
To conclude continuity of $\tilde h$ at the endpoints $e_j$,
observe that for every $\Lambda_0 \in \N$ there is a neighborhood 
$U$ of $G_j$ such that $\pi_0^{-1}(J) \cap W \subset U$ only if 
$\Lambda(W) \geq \Lambda_0$ (where the components $W$ and there maximal levels
$\Lambda(W)$ are as in the proof of Theorem~\ref{thm:cont}).
On the other hand, for every neighborhood $\tilde U$ of $e_i = \tilde h(e_j)$,
$\phi(W) \cap \tilde\pi_0^{-1}(\tilde c_i) \cap \tilde U \neq \emptyset$ 
only if $\Lambda(W)$ is sufficiently large.
Therefore $\tilde h$ maps small neighborhood of $e_j$ into small 
neighborhoods of $e_i$, and the continuity of $\tilde h$ at 
$\{ e_0, \dots, e_{N-1}\}$ follows.
\end{proof}

Recall that $\il ([0, 1], Q) = C \cup X$, where $X = \ILQ$ and $C$ is the ray 
containing $\bar 0$ that compactifies on $X$. 
Analogously, $\il ([0, 1], T) = \tilde C \cup \tilde X$. 
\begin{rem}\label{rem:H}
Let $H : \il ([0, 1], Q) \to \il ([0, 1], Q)$ be a homeomorphism. In a straightforward way it is 
possible to expand our construction of $\phi : X \to \tilde X$ to get the continuous map 
$\Phi : \il ([0, 1], Q) \to \il ([0, 1], T)$ such that $\Phi|_X = \phi$ and the map 
$\tilde H := \Phi \circ H \circ \Phi^{-1}$ is a well-defined homeomorphism on $\il ([0, 1], T)$.
Obviously $\tilde H|_{\tilde X} = \tilde h$.
\end{rem}
In \cite{BS} we have proved that every homeomorphism on  $\il ([0, 1], T)$ is isotopic to some power
of the shift map. Therefore, $\tilde H$ is isotopic to $\sigma^R$ for some $R \in \Z$. Since $\Phi$ 
is injective on $\il ([0, 1], Q) \setminus (\bigcup_{i=0}^{N-1} G_i)$ and 
$\Phi \circ H = \tilde H \circ \Phi$, $H$ restricted to 
$\il ([0, 1], Q) \setminus (\bigcup_{i=0}^{N-1} G_i)$ is pseudo-isotopic to $\sigma^R$, and $H$ permutes 
the $G_i$s in the same way as $\sigma^R$. 

\section{Isotopy}\label{sec:I}

Let $\ch_k$ denote a natural chain of $\il ([0, 1], Q)$. Then, by construction of $\Phi$, $\Phi(\ch_k)$ 
is a natural chain of $\il ([0, 1], T)$. Also, for every $k, l \in \N$, $\Phi$ maps the salient $k$-point 
of $C$ of $k$-level $l$ to the salient $k$-point of $\tilde C$ of the same $k$-level $l$.

Let $p, q$ be such that $H(\ch_p) \prec \ch_q$. Then, by construction of $\Phi$, 
$\tilde H(\Phi(\ch_p)) \prec \Phi(\ch_q)$. Let $\{ t_i : i \in \N \}$ denote all salient $p$-points of 
$C$ and $\{ s_i : i \in \N \}$ all salient $p$-points of $\tilde C$. Let $\{ t'_i : i \in \N \}$ denote 
all salient $q$-points of $C$ and $\{ s'_i : i \in \N \}$ all salient $q$-points of $\tilde C$.

Let $A_i$ be the maximal $p$-link-symmetric arc centered at $t_i$. Since $A_i$ is $p$-link-symmetric, 
and $H(\ch_{p}) \preceq \ch_{q}$, the image  $D_i := H(A_i)$ is $q$-link-symmetric and therefore has 
a well-defined central link  $\ell_q$, and a well-defined  center, we denote it as $m'_i$. In fact, 
$H(t_i)$ and $m'_i$ belong to the central link  $\ell_p$ and $m'_i$ is the $q$-point with the highest
$q$-level of all $q$-points of the arc component of $\ell_p$ which contains $H(t_i)$. We will write that 
$H(d) \approx b$ if $H(d)$ and $b$ belong not only to the same link, but to the same arc-component of 
that link. Thus $H(t_i) \approx m'_i$. To prove our main theorem we should show the following lemma:

\begin{lem}\label{shift}
There exists $R \in \Z$ such that $m'_i = t'_{i+R}$ for all sufficiently large integers $i \in \N$.
\end{lem}

The analogous statement 
for the tent map inverse limit spaces is Theorem 4.1 in \cite{BBS}, and its proof relies on properties 
of tent maps. Therefore, we will provide here a new proof for the inverse limits of quadratic maps. 

\begin{proof}
	Let $\Phi$ and $H$ be as in Remark \ref{rem:H}. Let $R$ be such that $\tilde H$ is isotopic to 
	$\sigma^R$.	Recall that $\Phi$ maps the salient $p$-point of $C$ of $p$-level $i$ to the salient 
	$p$-point of $\tilde C$ of the same $p$-level $i$, that is $\Phi(t_i) = s_i$, for all $i \in \N$, 
	and analogously for salient $q$-points, $\Phi(t'_i) = s'_i$, for all $i \in \N$. Since by 
	\cite[Theorem 4.12]{S} 
	and \cite[Theorem 4.1]{BBS}, $\tilde H(s_i) \approx s'_{i+R}$ for all $i$ sufficiently large, and 
	$\tilde H \circ \Phi = \Phi \circ H$, it follows that $H(t_i) \approx t'_{i+R}$ for all sufficiently 
	large $i \in \N$.
\end{proof}

To finish the proof of the main theorem we will invoke statements from \cite{BBS} and \cite{BS}. 
Although all of them are stated for tent maps inverse limits, their proofs work in the same way for the
inverse limits of quadratic maps.

In the same way as in the proof of \cite[Proposition 4.2]{BBS}, but using Lemma \ref{shift} instead of
\cite[Theorem 4.1]{BBS}, it follows that for every 
$p$-point $x \in C$ of $p$-level $i$ there exists a $q$-point $x'\in C$ of $q$-level 
$i+R$ such that $H(x) \approx x'$, for all sufficiently large integers $i \in \N$. Also, in the 
same way as in the proof of \cite[Theorem 1.3]{BBS}, it follows that the homeomorphism 
$H : \il ([0, 1], Q) \to \il ([0, 1], Q)$ is pseudo-isotopic to $\sigma^R$.

Note that the only places in the proofs of \cite{BS} that rely on properties of tent maps are those
where Theorem 4.1, Proposition 4.2 or Theorem 1.3 from \cite{BBS} are cited. Having now the analogous results proved for the quadratic family, the proof of our main theorem follows from \cite{BS}.

\section{Application to entropy}\label{sec:E}

In  \cite[Theorem 1.1]{BS1} we proved that for any self-homeomorphism 
$h : \il ([0, 1], T_s) \to \il ([0, 1], T_s)$, of the inverse limit space 
of a tent map with slope $s \in (\sqrt2, 2]$, the 
topological entropy $\htop(h) = |R| \log s$, where $R \in \Z$ is such that $h$ is isotopic 
to $\sigma^R$.

We call the fact that the homeomorphisms on a space $X$ can only take specific values 
{\em entropy rigidity}.  

If a quadratic map $Q_a$ has no non-trivial periodic intervals (that is, $Q_a$ is not
renormalizable), then the same theorem applies to a homeomorphism on its inverse limit 
space. Otherwise, the first return map to such a periodic interval is a new unimodal map, 
called the {\em renormalization} of the previous; thus we get a (possibly infinite)
sequence of nested cycles of periodic intervals, with periods $(p_i)_{i \ge 0}$, where $p_i$ 
divides $p_{i+1}$ and $p_0 = 1$. The effect of renormalization 
on the structure of the inverse limit space is well understood \cite{BaDi3}: The core of 
the inverse limit of $Q_a$ has $p_1$ proper subcontinua $G_k$, $k = 0, \dots , p_1-1$,
which are permuted cyclically by the shift homeomorphism, and each of them is homeomorphic to 
the inverse limit space of the renormalization. Outside $\cup_k G_k$, the core inverse limit
space has no other subcontinua than points and arcs. Hence, each homeomorphism $H$ of 
$\il ([0, 1], Q_a)$ can at most permute the subcontinua $G_k$, $k = 0, \dots , p_1-1$, in 
some way (isotopically to $\sigma^R$). 
In the proof of \cite[Theorem 1.2]{BS1} the erroneous suggestion was made, that within $G_k$, $k = 0, \dots , p_1-1$, 
the homeomorphism $H$ need not act isotopically to $\sigma^R$, 
and that ``it can still be arranged that $H$ maps $G_k$ to $G_{k+R \bmod p_1}$ isotopically 
to some other(!) power of $\sigma$'' (\cite[p.\ 999]{BS1}). However, the main result 
of this paper, Theorem \ref{main}, shows that this can not happen (hence, the rest 
of the proof of \cite[Theorem 1.2]{BS1} is superfluous). This leads to the corrected 
version of \cite[Theorem 1.2]{BS1}:
\begin{theorem}\label{thm:entropy-new}
Assume that $Q$ is a quadratic map with positive topological entropy and $\log s = \htop(Q)$.
	If $H$ is a homeomorphism on the inverse limit space $\il ([0, 1], Q)$, 
	then the topological entropy $\htop(H) = |R|\log s$,
	where $R \in \Z$ is such that $H$ is isotopic to $\sigma^R$.
\end{theorem}
The theorem has the same form as \cite[Theorem 1.1]{BS1} and can be proved in an analogous way.

\medskip
\noindent
Faculty of Mathematics, University of Vienna\\
Oskar Morgensternplatz 1, 1090 Wien, Austria\\
\texttt{henk.bruin@univie.ac.at}\\
\texttt{http://www.mat.univie.ac.at/}$\sim$\texttt{bruin/}\\[3mm]
Department of Mathematics, Faculty of Science, University of Zagreb\\
Bijeni\v cka 30, 10 000 Zagreb,
Croatia\\
\texttt{sonja@math.hr}\\
\texttt{http://www.math.hr/}$\sim$\texttt{sonja}

\end{document}